\documentclass[11pt]{amsart}
\usepackage{floatrow}
\usepackage[usenames,dvipsnames]{xcolor} 
\usepackage{amsmath}
\usepackage[utf8]{inputenc}
\usepackage{amsfonts}
\usepackage{amssymb}
\usepackage{tikz-cd}
\usepackage{graphicx}
\usepackage{bm}
\usepackage[numbers,sort]{natbib}
\usepackage{enumitem}
\usepackage[margin=1.33in]{geometry}
\usepackage[pdftex,
pdfpagemode=UseNone,
breaklinks=true,
extension=pdf,
colorlinks=true,
linkcolor=blue,
citecolor=PineGreen,
urlcolor=blue,
]{hyperref}
\usepackage{amsthm}

\theoremstyle{theorem}
\newtheorem{theorem}{Theorem}[section]
\newtheorem{question}[theorem]{Question}
\newtheorem{lemma}[theorem]{Lemma}
\newtheorem{corollary}[theorem]{Corollary}
\newtheorem{proposition}[theorem]{Proposition}

\theoremstyle{definition}

\newtheorem{definition}[theorem]{Definition}
\newtheorem{examplex}[theorem]{Example}
\newtheorem{notation}[theorem]{Notation}
\newenvironment{example}
{\pushQED{\qed}\examplex}
{\popQED\endexamplex}

\newtheorem{remarkx}[theorem]{Remark}
\newenvironment{remark}
{\pushQED{\qed}\remarkx}
{\popQED\endremarkx}

\DeclareMathOperator{\QQ}{\mathbb{Q}}
\DeclareMathOperator{\init}{In}

\DeclareMathOperator{\lord}{lord}
\DeclareMathOperator{\tord}{tord}
\DeclareMathOperator{\ord}{ord}

\DeclareMathOperator{\lex}{lex}
\DeclareMathOperator{\gr}{gr}

\DeclareMathOperator{\Znn}{\mathbb{Z}_{\geqslant 0}}

\usepackage{listings}

\definecolor{dkgreen}{rgb}{0,0.6,0}
\definecolor{gray}{rgb}{0.5,0.5,0.5}
\definecolor{mauve}{rgb}{0.58,0,0.82}

\begin{document}

\title[Multiplicity structure of the arc space of a fat point]{Multiplicity structure of the arc space\\ of a fat point}

\author{Rida Ait El Manssour}
\address{MPI MiS, Inselstra{\ss}e 22, 04103 Leipzig, Germany}
\email{rida.manssour@mis.mpg.de}

\author{Gleb Pogudin}
\address{LIX, CNRS, \'Ecole Polytechnique, Institute Polytechnique de Paris, 1 rue Honor\'e d'Estienne d'Orves, 91120, Palaiseau, France}
\email{gleb.pogudin@polytechnique.edu}

\begin{abstract}
  The equation $x^m = 0$ defines a fat point on a line. The algebra of regular functions on the arc space of this scheme is the quotient of $k[x, x', x^{(2)}, \ldots]$ by all differential consequences of $x^m = 0$.
This infinite-dimensional algebra admits a natural filtration by finite dimensional algebras corresponding to the truncations of arcs.
We show that the generating series for their dimensions equals $\frac{m}{1 - mt}$.
We also determine the lexicographic initial ideal of the defining ideal of the arc space.
These results are motivated by nonreduced version of the geometric motivic Poincar\'e series, multiplicities in differential algebra, and connections between arc spaces and the Rogers-Ramanujan identities.
We also prove a recent conjecture put forth by Afsharijoo in the latter context.
\end{abstract}

\maketitle

\section{Introduction}

\subsection{Statement of the main result}
Let $k$ be a field of characteristic zero. 
Consider an ideal $I \subset k[\mathbf{x}]$, where $\mathbf{x} = (x_1, \ldots, x_n)$, defining an affine scheme $X$.
We consider the polynomial ring
\[
k[\mathbf{x}^{(\infty)}] := k[x_i^{(j)} \mid 1\leqslant i \leqslant n,\; j \geqslant 0]
\]
in infinitely many variables $\{x_i^{(j)} \mid 1\leqslant i \leqslant n,\; j \geqslant 0\}$.
This ring is equipped with a $k$-linear derivation $a \mapsto a'$ defined on the generators by
\[
  (x_i^{(j)})' = x_i^{(j + 1)} \text{ for } 1\leqslant i \leqslant n,\; j \geqslant 0.
\]
Then we define the ideal $I^{(\infty)} \subset k[\mathbf{x}^{(\infty)}]$ of the arc space of $X$ by
\[
  I^{(\infty)} := \langle f^{(j)} \mid f \in I, \; j \geqslant 0 \rangle.
\]

In this paper, we will focus on the case of a fat point $\mathcal{I}_m := \langle x^m \rangle \subset k[x]$ of multiplicity $m \geqslant 2$.
Although the zero set of $\mathcal{I}_m^{(\infty)}$ over $k$ consists of a single point with all the coordinates being zero, 
the dimension of the corresponding quotient algebra $k[x^{(\infty)}] / \mathcal{I}_m^{(\infty)}$ (the ``multiplicity'' of the arc space) is infinite.

One can obtain a finer description of the multiplicity structure of $k[x^{(\infty)}] / \mathcal{I}_m^{(\infty)}$ by considering its filtration by finite-dimensional algebras induced by the truncation of arcs
\[
  k[x^{(\leqslant \ell)}] / \mathcal{I}_m^{(\infty)} := k[x^{(\leqslant \ell)}] / ( k[x^{(\leqslant \ell)}] \cap \mathcal{I}_m^{(\infty)}),
\]
where $x^{(\leqslant \ell)} := \{x, x', \ldots, x^{(\ell)}\}$, and arranging the dimensions of these algebras into a generating series
\begin{equation}\label{eq:gen_ser}
  D_{\mathcal{I}_m}(t) := \sum\limits_{\ell = 0}^\infty \dim_k(k[x^{(\leqslant \ell)}] / \mathcal{I}_m^{(\infty)}) \cdot t^\ell.
\end{equation}
The main result of this paper is
\begin{equation}\label{eq:main_result}
  D_{\mathcal{I}_m}(t) = \frac{m}{1 - mt}.
\end{equation}

\subsection{Motivations and related results}

Our motivation for studying the series~\eqref{eq:gen_ser} comes from three different areas: algebraic geometry, differential algebra, and combinatorics.

\begin{itemize}[leftmargin=8pt]
    \item From the point of view of the \emph{algebraic geometry}, $I^{(\infty)}$ defines the \emph{arc space $\mathcal{L}(X)$}~\cite{ArcLecture} of the scheme $X$.
    Geometrically, the points of the arc space correspond to the Taylor coefficients of the $k[\![t]\!]$-points of $X$.
    The arc space of a variety can be viewed as an infinite-order generalization of the tangent bundle or the space of formal trajectories on the variety.
    For properties and applications of arc spaces, we refer to~\cite{ArcLecture, arcsbook}.
    
    The \emph{reduced structure} of an arc space is often described by means of the geometric motivic Poincar\'e series~\cite[\S~2.2]{ArcLecture}:
    \begin{equation}\label{eq:geompoincare}
      P_X(t) := \sum\limits_{\ell = 0}^\infty [\pi_\ell(\mathcal{L}(X))] \cdot t^\ell,
    \end{equation}
    where $\pi_\ell$ denotes the projection of $\mathcal{L}(X)$ to the affine subspace with the coordinates $\mathbf{x}^{(\leqslant \ell)}$ (i.e., the truncation at order $\ell$) and $[Z]$ denotes the class of variety $Z$ in the Grothendieck ring~\cite[\S~2.3]{ArcLecture}.
    A fundamental result about these series is the Denef-Loeser theorem~\cite[Theorem~1.1]{DeLoRational} saying that $P_X(t)$ is a rational power series.
    
    The arc spaces may also have a rich \emph{scheme (i.e.. nilpotent) structure} (see~\cite{linshaw2021standard, FM2018, dumanski2021reduced}) reflecting the geometry of the original scheme~\cite{S11, BH2020}.
    In the case of a fat point $\mathcal{I}_m = \langle x^m \rangle \subset k[x]$, we will have $\pi_\ell(\mathcal{L}(X)) \cong \mathbb{A}^0$, so the geometric motivic Poincar\'e series equal to
    \[
      P(t) = \frac{[\mathbb{A}^0]}{1 - t},
    \]
    where $[\mathbb{A}^0]$ is the class of a point.
    Note that the series does not depend on the multiplicity $m$ of the point.
    One way to capture the scheme structure of $\mathcal{L}(X)$ could be to take the components of the projections in~\eqref{eq:geompoincare} with their multiplicities.
    For example, for the case $\mathcal{I}_m$, one will get
    \[
        \sum\limits_{\ell = 0}^\infty \dim_k \bigl(k[x^{(\leqslant \ell)}] / \mathcal{I}_m^{(\infty)}\bigr) \cdot [\mathbb{A}^0] \cdot t^\ell = D_{\mathcal{I}_m}(t) [\mathbb{A}^0].
    \]
    Our result~\eqref{eq:main_result} implies that the series above is rational as in the Denef-Loeser theorem.
    Interestingly, the shape of the denominator is different from the one in~\cite[Theorem~2.2.1]{ArcLecture}.
    The formula above is not the only way to take the multiplicties into account, a related more popular approach is via Arc Hilbert-Poincar\'e series~\cite[\S 9]{handbook} (see also~\cite{M13, Bruschek2012}).
    
    \item \emph{Differential algebra} studies, in particular, differential ideals in $k[\mathbf{x}^{(\infty)}]$, that is, ideals closed under derivation.
    From this point of view, $I^{(\infty)}$ is the differential ideal generated by $I$.
    Understanding the structure of the differential ideals $\mathcal{I}_m^{(\infty)}$ is a key component of the low power theorem~\cite{lowpower, lowpower2} which provides a constructive way to detect singular solutions of algebraic differential equations in one variable.
    Besides that, various combinatorial properties of $\mathcal{I}_m^{(\infty)}$ have been studied in differential algebra, see~\cite{Keefe1960, x2, vertex, Zobnin2005, Zobnin2009, RidaAnnaLaura}.
    
    While there is a rich dimension theory for solution sets of systems of algebraic differential equations~\cite{diffpimpoly, ranks, koldim}, we are not aware of a notion of multiplicity of a solution of such a system.
    In particular, the existing differential analogue of the B\'ezout theorem~\cite{diffBezout} provides only a bound unlike the equality in classical B\'ezout theorem~\cite[Theorem~7.7, Chapter 1]{hartshorne}.
    Our result~\eqref{eq:main_result} suggests that one possibility is to define the multiplicity of a solution as the growth rate of multiplicities of its truncations, and this definition will be consistent with the usual algebraic multiplicity for the case of a fat point on a line.
    
    \item Connections between the multiplicity structure of the arc space of a fat point and Rogers-Ramanujan partition identities from \emph{combinatorics} were pointed out by Bruschek, Mourtada, and Schepers in~\cite{Bruschek2012} (for a recent survey, see~\cite[\S 9]{handbook}).
    In particular, they used Hilbert-Poincare series of similar nature to~\eqref{eq:gen_ser} (motivated by the singularity theory~\cite[Section~4]{M13}) for obtaining new proofs of the Rogers-Ramanujan identities and their generalizations.
    In this direction, new results have been obtained recently in~\cite{Afsharijoo2021, afsharijoo2021andrewsgordon,Bai2019}.
    In~\cite{Afsharijoo2021}, Afsharijoo used computational experiments to conjecture~\cite[Section~5]{Afsharijoo2021} the initial ideal of $\mathcal{I}_{m}^{(\infty)}$ with respect to the weighted lexicographic ordering (a special case was already conjectured in~\cite[Section~1]{AM2020}).
    This conjecture would imply a new set of partition identities~\cite[Conjecture~5.1]{Afsharijoo2021}.
    Using combinatorial techniques, some of them have been proved in~\cite{Afsharijoo2021}, and the rest were established in~\cite{afsharijoo2021andrewsgordon} (see also~\cite{lotharing}).
    However, the original algebraic conjecture about $\mathcal{I}_m^{(\infty)}$ remained open.
    As a byproduct of our proof of~\eqref{eq:main_result}, we prove this conjecture (see Theorem~\ref{thm:non-initial}) thus giving a new proof of one of the main results of~\cite{afsharijoo2021andrewsgordon}.
\end{itemize}

Understanding the structure of ideal $\mathcal{I}_m^{(\infty)}$ is known to be challenging: for example, its Gr\"obner basis with respect to the lexicographic ordering is not just infinite but even differentially infinite~\cite{Zobnin2005, AM2020}, and the question about the nilpotency index of $x_i^{(j)}$'s modulo $\mathcal{I}_m^{(\infty)}$ posed by Ritt in 1950~\cite[Appendix, Q.5]{Ritt} remained open for sixty years until~\cite{x2} (see also~\cite{Keefe1960, vertex}).

The statement~\eqref{eq:main_result} appeared in the Ph.D. thesis~\cite[Theorem~3.4.1]{GlebsPhd} of the second author but the proof given there was incorrect.
We are grateful to Alexey Zobnin for pointing out the error.
The proof presented in this paper uses different ideas than the erroneous proof in~\cite{GlebsPhd}.

\textbf{Update (20 February, 2024).} We would like to thank Ilya Dumanski for pointing out that the main dimension result~\eqref{eq:main_result} could also be deduced from a combination of Propositions 2.1 and 2.3 from~\cite{Feigin2002}.


\subsection{Overview of the proof}

The key technical tool used in our proofs is a representation of the quotient algebra $k[x^{(\infty)}] / \mathcal{I}_m^{(\infty)}$ as a subalgebra in certain differential exterior algebra constructed in~\cite{x2} (see Section~\ref{sec:exterior}).
 The injectivity of this representation builds upon the knowledge of a Gr\"obner basis for $\mathcal{I}_m^{(\infty)}$ with respect to degree reverse lexicographic ordering~\cite{Bruschek2012, Zobnin2009, lowpower}.
We approach the equality~\eqref{eq:main_result} as a collection of inequalities
\begin{equation}\label{eq:bounds}
  m^{\ell + 1} \leqslant \dim_k (k[x^{(\leqslant \ell)}] / \mathcal{I}_m^{(\infty)}) \leqslant m^{\ell + 1} \quad\text{ for every } \ell \geqslant 0,\; m \geqslant 1.
\end{equation}
The starting point of our proof of \emph{the lower bound} uses the insightful conjecture by Afsharijoo~\cite[Section~5]{Afsharijoo2021} suggesting how the standard monomials of $\mathcal{I}_m^{(\infty)}$ with respect to the lexicographic ordering look like.
Using the exterior algebra representation, we prove that these monomials are indeed linearly independent modulo $\mathcal{I}_m^{(\infty)}$, and deduce the lower bound from this (Section~\ref{sec:combi} and~\ref{sec:lower}).

In order to prove \emph{the upper bound} from~\eqref{eq:bounds}, we represent the image of $k[x^{(\leqslant\ell)}] / \mathcal{I}_m^{(\infty)}$ in the differential exterior algebra as a deformation of an algebra which splits as a direct product of $\ell + 1$ algebras of dimension $m$ thus yielding the desired upper bound (Section~\ref{sec:upper}).


\subsection{Structure of the paper}

The rest of the paper is organized as follows.
Section~\ref{sec:prelim} contains definitions and notations used to state the main results.
Section~\ref{sec:main} contains the main results of the paper. 
The proofs of the results are given in Section~\ref{sec:proofs}.
Section~\ref{sec:compute} describes computational experiments in {\sc Macaulay2} we performed to check whether formulas similar to~\eqref{eq:main_result} hold for more general fat points in $k^n$.
We formulate some open questions based on the results of these experiments.


\section{Preliminaries}\label{sec:prelim}

Definitions~\ref{def:diffrings}-\ref{def:diffideals} provide necessary background in differential algebra. 
For further details, we refer to~\cite[Chapter~1]{Kaplansky} or~\cite[\S I.1, I.2]{Kolchin}.

\begin{definition}[Differential rings and fields]\label{def:diffrings}
  A {\em differential ring} $(R,\,')$ is a commutative ring with a derivation $'\!\!:R\to R$, that is, a map such that, for all $a,b\in R$, $(a+b)'=a'+b'$ and $(ab)'=a'b+ab'$. 
  A {\em differential field} is a differential ring that is a field.
  For  $i>0$,  $a^{(i)}$ denotes the $i$-th order derivative of $a \in R$.
\end{definition}

\begin{notation}
  Let $x$ be an element of a differential ring and $h \in \Znn$. We introduce
  \[
    x^{(<h)} := (x, x', \ldots, x^{(h - 1)}) \quad \text{ and }\quad x^{(\infty)} := (x, x', x'', \ldots).
  \]
  $x^{(\leqslant h)}$ is defined analogously.
  If $\mathbf{x} = (x_1, \ldots, x_n)$ is a tuple of elements of a differential ring, then
  \[
    \mathbf{x}^{(< h)} := (x_1^{(< h)}, \ldots, x_n^{(<h)}) \quad\text{ and }\quad \mathbf{x}^{(\infty)} := (x_1^{(\infty)}, \ldots, x_n^{(\infty)}).
  \]
\end{notation}

\begin{definition}[Differential polynomials]
  Let $R$ be a differential ring. 
  Consider a ring of polynomials in infinitely many variables
  \[
  R[x^{(\infty)}] := R[x, x', x'', x^{(3)}, \ldots]
  \]
  and extend the derivation from $R$ to this ring by $(x^{(j)})' := x^{(j + 1)}$.
  The resulting differential ring is called \emph{the ring of differential polynomials in $x$ over $R$}.
  The ring of differential polynomials in several variables is defined by iterating this construction.
\end{definition}

\begin{definition}[Differential ideals]\label{def:diffideals}
   Let $S : =R[x_1^{(\infty)}, \ldots, x_n^{(\infty)}]$  be a ring of differential polynomials over a differential ring $R$.
   An ideal $I \subset S$ is called \emph{a differential ideal} if $a' \in I$ for every $a \in I$.
  
   One can verify that, for every $f_1, \ldots, f_s \in S$, the ideal
   \[
      \langle f_1^{(\infty)}, \ldots, f_s^{(\infty)} \rangle
    \]
    is a differential ideal.
    Moreover, this is the minimal differential ideal containing $f_1, \ldots, f_s$, and we will denote it by $\langle f_1, \ldots, f_s \rangle^{(\infty)}$.
\end{definition}

\begin{definition}[Fair monomials]\label{def:fair}
\begin{itemize}
    \item[]
    \item For a monomial $m = x^{(h_0)}x^{(h_1)}\cdots x^{(h_\ell)} \in k[x^{(\infty)}]$, we define \emph{the order} and \emph{lowest order} as $\ord m := \max\limits_{0 \leqslant i \leqslant \ell} h_i$ and $\lord m := \min\limits_{0 \leqslant i \leqslant \ell} h_i$, respectively.
    \item A monomial $m \in k[x^{(\infty)}]$ is called \emph{fair} (resp., \emph{strongly fair}) if 
    \[
    \lord m \geqslant \deg m - 1 \quad\text{(resp., $\lord m \geqslant \deg m$)}.
    \]
    We denote the sets of all fair and strongly fair monomials by $\mathcal{F}$ and $\mathcal{F}_s$, respectively.
    By convention, $1 \in \mathcal{F}$ and $1 \in \mathcal{F}_s$.
    Note that $\mathcal{F}_s \subset \mathcal{F}$.
    
    \item For every integers $a, b \geqslant 0$,
    we define
    \[
        \mathcal{F}_{a, b} := \mathcal{F}^{a} \cdot \mathcal{F}_s^{b},
    \]
    where the product of sets of monomials is the set of pairwise products.
    In other words, $\mathcal{F}_{a, b}$ is a set of all monomials representable as a product of $a$ fair monomials and $b$ strongly fair monomials.
\end{itemize}
\end{definition}

\begin{remark}
The notion of fair monomials was inspired from the conjectured construction of the initial ideal of $\langle x^i , (x^m)^{(\infty)} \rangle$ given in  \cite[conjecture 5.1]{Afsharijoo2021}.
We use the notion to formulate concisely and prove the conjecture (see Theorem~\ref{thm:non-initial}).
\end{remark}

\begin{example}
  The monomials of order at most two in $\mathcal{F}$ and $\mathcal{F}_s$ are:
  \begin{align*}
    \mathcal{F} \cap k[x^{(\leqslant 2)}] &= \{1,\; x,\; x',\; (x')^2,\; x'x'',\; x'',\; (x'')^2,\; (x'')^3\},\\
    \mathcal{F}_s \cap k[x^{(\leqslant 2)}] &= \{1, x', x'', (x'')^2\}.
  \end{align*}
  Using this, one can produce the monomials of order at most one in $\mathcal{F}_{1, 1}$ and $\mathcal{F}_{2, 0}$
  \begin{align*}
      \mathcal{F}_{1, 1} \cap k[x^{(\leqslant 1)}] &= \{1, x, xx', x', (x')^2, (x')^3\},\\
      \mathcal{F}_{2, 0} \cap k[x^{(\leqslant 1)}] &= \{1, x, x^2, xx', x(x')^2, x', (x')^2, (x')^3, (x')^4\}\qedhere
  \end{align*}
  For example, $(x')^3 \in \mathcal{F}_{1, 1}$ can be written as $(x')^2 \cdot x'$, where $(x')^2 \in \mathcal{F}$ and $ x' \in \mathcal{F}_s$.
  Likewise, for the monomials of order at most two, we can write:
  \begin{align*}
      \mathcal{F}_{1, 1} \cap k[x^{(\leqslant 2)}] = &\{1, x, x', x'', xx', xx'', (x')^2, x'x'', (x'')^2, x(x'')^2, (x')^3, (x')^2x'',\\ &x'(x'')^2, (x'')^3, (x')^2 (x'')^2, x'(x'')^3, (x'')^4, (x'')^5\}.
  \end{align*}
\end{example}


\section{Main results}\label{sec:main}

The algebra of regular functions on the arc space of a fat point $x^m = 0$ admits a natural filtration by subalgebras induced by the truncation of arcs.
Our first main result, Theorem~\ref{thm:dimension}, gives a simple formula for the dimension of the subalgebra induced by the truncation at order $h$.
Corollary~\ref{cor:gen_ser} gives the generating series for these dimensions (as in~\eqref{eq:main_result}).

\begin{theorem}\label{thm:dimension}
  Let $m$ and $h$ be positive integers and $k$ be a differential field of zero characteristic. Then
  \[
      \dim_k \bigl( k[x^{(\leqslant h)}] / (k[x^{(\leqslant h)}] \cap \langle x^m\rangle^{(\infty)}) \bigr) = m^{h + 1}.
  \]
\end{theorem}
\begin{corollary}\label{cor:gen_ser}
   Let $m$ be a positive integer and $k$ be a differential field of zero characteristic.
   Then
   \[
      \sum\limits_{\ell = 0}^\infty \dim_k(k[x^{(\leqslant \ell)}] / \langle x^m\rangle^{(\infty)})\cdot t^\ell = \frac{m}{1 - mt},
   \]
   where $k[x^{(\leqslant \ell)}] / \langle x^m\rangle^{(\infty)} := k[x^{(\leqslant \ell)}] / (k[x^{(\leqslant \ell)}] \cap \langle x^m\rangle^{(\infty)})$.
\end{corollary}

Given a polynomial ideal and monomial ordering, the monomials which do not appear as leading terms of the elements of the ideal are called \emph{standard} monomials.
Motivated by applications to combinatorics, Afsharijoo used computations experiment to conjecture~\cite[Section 5]{Afsharijoo2021} a description of the standard monomials of $\langle x^m\rangle^{(\infty)}$ with respect to the degree lexicographic ordering.
Our second main result, Theorem~\ref{thm:non-initial}, gives such a description and, combined with Lemma~\ref{lem:nonoverlap}, establishes the conjecture.

\begin{theorem}\label{thm:non-initial}
    Let $k$ be a differential field of zero characteristic.
    Consider a degree lexicographic monomial ordering on $k[x^{(\infty)}]$ with the variables oredered as $x < x' < x'' < \ldots$.
    Let $m$ and $i$ be positive integers with $1 \leqslant i \leqslant m$.
    Then the set of standard monomials of the ideal $\langle x^i, (x^m)^{(\infty)} \rangle$ is $\mathcal{F}_{i - 1, m - i}$ (see Definition~\ref{def:fair}).
    Note that, for $i = m$, we obtain differential ideal $\langle x^m \rangle^{(\infty)}$.
\end{theorem}

\begin{corollary}\label{cor:non-initial}
    Theorem~\ref{thm:non-initial} also holds for the following orderings:
    \begin{itemize}
        \item \emph{purely lexicorgraphic} with the variables ordered as in Theorem~\ref{thm:non-initial};
        \item \emph{wieghted lexicographic}: monomials are first compared by the sum of the orders and then lexicographically as in Theorem~\ref{thm:non-initial}.
    \end{itemize}
\end{corollary}

\begin{remark}
The multiplicity of the scheme of polynomial arcs of degree less than $h$ of $x = 0$ (defined by $\langle x^m, x^{(h)} \rangle^{(\infty)}$) has been studied in~\cite{RidaAnnaLaura}.
It was shown~\cite[Theorem~2.5]{RidaAnnaLaura} that this multiplicity (equal to $\dim_k k [ x^{( \infty)}] / \langle x^m, x^{(h)} \rangle^{(\infty)}$) is a polynomial in $m$ of degree $h$ which is the Erhart polynomial of some lattice polytope.
Theorem~\ref{thm:dimension} together with a natural surjective morphism $k[x^{(< h)}] / \langle x^m\rangle^{(\infty)} \to k [ x^{( \infty)}] / \langle x^m, x^{(h)} \rangle^{(\infty)}$ implies that this polynomial is bounded by $m^h$.
\end{remark}


\section{Proofs}\label{sec:proofs}

\subsection{Key technical tool: embedding into exterior algebra}\label{sec:exterior}

\begin{notation}
  Let $k$ be a field.
  Then, for $\bm{\xi} = (\xi_0, \xi_1, \ldots, \xi_n)$ we introduce a countable collection of symbols $\{\xi_i^{(j)} \mid 0 \leqslant i \leqslant n, \; j \geqslant 0\}$ and by $\Lambda_k(\bm{\xi}^{(\infty)})$ denote the exterior algebra of a $k$-vector space spanned by these symbols.
  $\Lambda_k(\bm{\xi}^{(\infty)})$ is equipped with a structure of a (noncommutative) differential algebra by 
  \[
    (\xi_j^{(i)})' := \xi_j^{(i + 1)} \quad \text{for every } i\geqslant 0 \text{ and } 0 \leqslant j \leqslant n. 
  \]
\end{notation}

The following proposition is a minor modification of~\cite[Lemma~1]{x2}.
The proof we will give is a simplification of the proof of~\cite[Lemma~1]{x2} which will be extended to a proof of Lemma~\ref{lem:ext_extra}.

\begin{proposition}\label{prop:ext_embedding}
  Let $m$ be a positive integer, and consider tuples $\bm{\eta} = (\eta_0, \ldots, \eta_{m - 2})$ and $\bm{\xi} = (\xi_0, \ldots, \xi_{m - 2})$.
  Let 
  \[
  \Lambda := \Lambda_k(\bm{\eta}^{(\infty)}) \otimes \Lambda_k(\bm{\xi}^{(\infty)}),
  \]
  which is equipped with a structure of differential algebra (as a tensor product of differential algebras, using the Leibnitz rule, that is $(a \otimes b)' := a'\otimes b + a \otimes b'$).
  Consider a differential homomorphism $\varphi \colon k[x^{(\infty)}] \to \Lambda$ defined by
  \[
    \varphi(x) := \sum\limits_{i = 0}^{m - 2} \eta_i \otimes \xi_i.
  \]
  Then the kernel of $\varphi$ is $\langle x^m \rangle^{(\infty)}$.
\end{proposition}

\begin{example}
  Consider the case $m = 3$.
  Then we will have
  \[
      \varphi(x) = \eta_0 \otimes \xi_0 + \eta_1 \otimes \xi_1.
  \]
  The image of $x'$ will be then
  \[
      \varphi(x') = (\varphi(x))' = \eta_0' \otimes \xi_0 + \eta_0 \otimes \xi_0' + \eta_1' \otimes \xi_1 + \eta_1 \otimes \xi_1'.
  \]
  One can show, for example, that $(x')^4 \not\in \langle x^3\rangle^{(\infty)}$ by showing that $\varphi((x')^4) \neq 0$:
  \[
    \varphi((x')^4) = 24 (\eta_0 \wedge \eta_0' \wedge \eta_1 \wedge \eta_1') \otimes (\xi_0\wedge \xi_0' \wedge \xi_1 \wedge \xi_1') \neq 0.
  \]
  Furthermore, a direct computation shows that $\varphi((x')^5) = 0$.
  Combined with Proposition~\ref{prop:ext_embedding}, this implies $(x')^5 \in \langle x^3 \rangle^{(\infty)}$.
\end{example}

\begin{proof}[Proof of Proposition~\ref{prop:ext_embedding}]
  Consider $(\varphi(x))^m$. This is a sum of tensor products of exterior polynomials of degree $m$ in $m - 1$ variables, so it must be zero.
  Since $(\varphi(x))^m = 0$ and $\varphi$ is a differential homomrphism, we conclude that $\operatorname{Ker}\varphi \supset \langle x^m \rangle^{(\infty)}$.
  
  Now we will prove the inverse inclusion.
  We define the weighted degree inverse lexicographic ordering $\prec$ on $k[x^{(\infty)}]$ (cf.~\cite[p. 524]{Zobnin2009}): $M \prec N$ if an only if
  \begin{itemize}[itemsep=0pt]
      \item $\tord M < \tord N$, where $\tord$ is defined as the sum of the orders, or
      \item $\tord M = \tord N$ and $\deg M < \deg N$ or
      \item $\tord M = \tord N$, $\deg M = \deg N$, and $N$ is lexicographically \emph{lower} than $M$, where the variables are ordered $x < x' < x'' < \ldots$.
  \end{itemize}
  For example, we will have $x \prec x' \prec x'' \prec \ldots$ and $xx'' \prec (x')^2$.
  Then, for every $h \geqslant 0$, the leading monomial of $(x^m)^{(h)}$ with respect to $\prec$ is $(x^{(q)})^{m - r} (x^{(q + 1)})^r$, where $q$ and $r$ are the quotient and the reminder of the integer division of $h$ by $m$, respectively.
  Let $\mathcal{M}$ be the set of all monomials not divisible by any monomial of the form $(x^{(q)})^{m - r} (x^{(q + 1)})^r$.
  Then we can characterize $\mathcal{M}$ as follows:
  \[
    \mathcal{M} = \{x^{(h_0)} \ldots x^{(h_\ell)} \mid h_0 \leqslant \ldots \leqslant h_\ell, \; \forall\, 0\leqslant i \leqslant \ell - m + 1 \colon h_{i + m - 1} > h_i + 1\}.
  \]
  We will define a linear map $\psi$ from $\mathcal{M}$ to the set of monomials in $\Lambda$ with the following properties:
  \begin{enumerate}[label=(P\arabic*)]
      \item\label{prop:nonzero} \emph{For every $P \in \mathcal{M}$, $\psi(P) \neq 0$.}
      
      \item\label{prop:ord_prev} \emph{For every $P \in \mathcal{M}$, the monomial $\psi(P)$ appears in the polynomial $\varphi(P)$ and, for any $P_0 \in \mathcal{M}$ such that $P_0 \prec P$, $\psi(P)$ does not appear in the polynomial $\varphi(P_0)$.}
  \end{enumerate}
  Informally speaking, $\psi(M)$ is the ``leading monomial'' in $\varphi(M)$.
  Once such a map $\psi$ has been defined, we can prove the proposition as follows.
  Let $Q \in \operatorname{Ker}\varphi \setminus \langle x^m \rangle^{(\infty)}$.
  By replacing $Q$ with the result of the reduction of $Q$ by $x^m, (x^m)', \ldots$ with respect to $\prec$, we can further assume that all the monomials in $Q$ belong to $\mathcal{M}$\footnote{Interestingly, although it is known that $x^m, (x^m)', \ldots$ form a Gr\"obner basis, we do not really need to use this fact here since a reduction with respect to any set of polynomials is well-defined.}.
  Let $Q_0$ be the largest of them. 
  By~\ref{prop:nonzero} and~\ref{prop:ord_prev}, $\varphi(Q_0)$ will involve $\psi(Q_0)$ and $\varphi(Q - Q_0)$ will not, so $\varphi(Q) \neq 0$.
  This contradicts the assumption that $Q \in \operatorname{Ker}\varphi$. The proposition is proved.
  
  Therefore, it remains to define $\psi$ satisfying~\ref{prop:nonzero} and~\ref{prop:ord_prev}.
  We will start with the case $m = 2$ to show the main idea while keeping the notation simple.
  We define $\psi$ by
  \begin{equation}\label{eq:def_psi2}
    \psi(x^{(h_0)} \ldots x^{(h_\ell)}) :=  \left( \eta^{(0)} \otimes \xi^{(h_0)} \right) \wedge \left( \eta^{( 1)} \otimes \xi^{(h_1 - 1)} \right) \wedge \cdots \wedge \left( \eta^{(\ell)} \otimes \xi^{(h_\ell - \ell)} \right),
  \end{equation}
  where $h_0 \leqslant h_1 \leqslant \ldots \leqslant h_\ell$.
  For proving~\ref{prop:nonzero}, we observe that, if $h_{i + 1} > h_i + 1$ for all $i$, then $h_0 < h_1 - 1 < h_2 -2 < \ldots < h_\ell - \ell$, so there are no coinciding $\xi$'s in~\eqref{eq:def_psi2}.
  The construction for arbitrary $m$ will consist of splitting the monomial into $m - 1$ interlacing submonomials and applying~\eqref{eq:def_psi2} with $(\eta_i, \xi_i)$ to $i$-th submonomial. 
  More formally, if $h_0 \leqslant h_1 \leqslant \ldots  \leqslant h_\ell$, we define
  \begin{equation}\label{eq:def_psi}
    \psi(x^{(h_0)} \ldots x^{(h_\ell)}) := \prod\limits_{i = 0}^\ell \left( \eta_{i \,\%\, (m - 1)}^{([i / (m - 1)])} \otimes \xi_{i \,\%\, (m - 1)}^{(h_i - [i / (m - 1)])} \right),
  \end{equation}
  where $a \;\%\; b$ denotes the remainder of the division of $a$ by $b$, and $[\alpha]$ denotes the integer part of $\alpha$.
  \ref{prop:nonzero} is proved by applying~\ref{prop:nonzero} for $m = 2$ to each submonomial.
   
  For proving~\ref{prop:ord_prev}, 
  consider $P_0 \in \mathcal{M}$ with $P_0 \preceq P$ and $\psi(P)$ appearing in $\varphi(P_0)$.
  Since $\psi$ preserves the total orders and doubles the degrees, we have $\tord P_0 = \tord P$ and $\deg P_0 = \deg P$.
  Let $H := \ord P_0$.
  Since $P_0 \preceq P$, we have $H \geqslant h_\ell$.
  Since the maximal orders of $\eta$ and $\xi$ in $\psi(P)$ do not exceed $[\ell / (m - 1)]$ and  $h_\ell - [\ell / (m - 1)]$, respectively, we have $H \leqslant h_\ell$.
  Thus, $H = h_\ell$.
  Applying the same argument recursively to $P / x^{(h_\ell)}$ and $P_0 / x^{(h_\ell)}$, we conclude that $P = P_0$.
      
   We will prove that $\varphi(P)$ involves $\psi(P)$ by induction on $\deg P$.
   The case $\deg P = 0$ is clear.
   Consider $P$ with $\deg P > 0$.
   Similarly to the preceding argument,
   one can obtain $\psi(P)$ (from $\psi(P/ x^{(\ell)})$) only by taking $\eta_{\ell \,\%\, (m - 1)}^{([\ell / (m - 1)])} \otimes \xi_{\ell \,\%\, (m - 1)}^{(h_\ell - [\ell / (m - 1)])}$ (i.e., the last term in~\eqref{eq:def_psi}) from one of the occurrences of $x^{(h_\ell)}$ in $P$. 
   Therefore, the coefficient in front of $\psi(P)$ in $\varphi(P)$ will be up to sign equal to $\deg_{x^{(h_\ell)}}$ times the coefficient in front of $\psi(P/x^{(h_\ell)})$ in $\varphi(P/x^{(h_\ell)}).$
   The latter is nonzero by the induction hypothesis.
   \end{proof}
  
\begin{lemma}\label{lem:ext_extra}
  In the notation of Proposition~\ref{prop:ext_embedding}, let $1 \leqslant r < m$.
  Then the preimage of the ideal in $\Lambda$ generated by $\eta_{r-1} \otimes \xi_{r-1},  \ldots, \eta_{m  - 2} \otimes \xi_{m  - 2}$ under $\varphi$ is equal to $\langle (x^m)^{(\infty)}, x^{r} \rangle$.
\end{lemma}

\begin{proof}
We first prove that the image of $x^r$ belongs to $\langle \eta_{r-1} \otimes \xi_{r - 1},  \ldots, \eta_{m  - 2} \otimes \xi_{m  - 2} \rangle$. This is because $\varphi( x^r )$ is the sum of monomials which are products of $r$ different $\eta_i \otimes \xi_i$'s. Since there are $m-1$ of them, every such monomial will involve at least one of the last $m-r$ of $\eta_i \otimes \xi_i$'s.

Let us consider a polynomial $g \in k[x^{(\infty)}] \setminus \langle (x^m)^{ (\infty)}, x^r \rangle $ and prove that $\varphi(g)$ does not belong to $\langle \eta_{r - 1} \otimes \xi_{r - 1},  \ldots, \eta_{m  - 2} \otimes \xi_{m  - 2} \rangle$. 
We can assume that each monomial $P$ of $g$ belongs to 
 \[
    \mathcal{M}_r = \{M \in \mathcal{M} \mid \deg_x M < r \text{ or } 0 < h_{r - 1}\}.
  \]
We will use the map $\psi$ defined in \eqref{eq:def_psi}. 
In fact, $\psi(P)$ does not involve the zero-order derivatives of $\xi_{r - 1}, \ldots, \xi_{m - 2}$ since $h_i - [i / (m - 1)]$ can only be zero for a monomial in $\mathcal{M}$ only if $i \leqslant r - 2$.
Thus
\[
\psi(P) \not\in \langle \eta_{r - 1} \otimes \xi_{r - 1},  \ldots, \eta_{m - 2} \otimes \xi_{m - 2} \rangle.
\]

Assume that $P_0$ is the largest summand that appears in $g$.
Then $\varphi(P_0)$ involves $\psi(P_0)$ but $\varphi(g - P_0)$ does not. 
Therefore, $\varphi(g)$ does not belong to $\langle \eta_{r - 1} \otimes \xi_{r - 1},  \ldots, \eta_{m - 2} \otimes \xi_{m - 2} \rangle$.
\end{proof}


\subsection{Upper bounds for the dimension}\label{sec:upper}

Throughout the section, we fix a differential field $k$ of zero characteristic.

\begin{proposition}\label{prop:dim_bound}
  Let $m, h$ be positive integers.
  By $A_{m, h}$ we denote the subalgebra of $k[x^{(\infty)}] / \langle x^m\rangle^{(\infty)}$ generated by the images of $x, x', \ldots, x^{(h)}$.
  Then
  \[
    \dim A_{m, h} \leqslant m^{h + 1}.
  \]
\end{proposition}

First we describe a general construction which will be a special case of the so-called associated graded algebra.
Let $A = A_0 \oplus A_1 \oplus A_2 \oplus \ldots$ be a $\mathbb{Z}_{\geqslant 0}$-graded algebra over $k$ equipped with a homogeneous derivation of weight one (that is, $A_i' \subseteq A_{i + 1}$ for every $i \geqslant 0$).
We introduce a map $\gr\colon A \to A$ defined as follows.
Consider a nonzero $a \in A$ and let $i$ be the largest index such that $a \in A_i \oplus A_{i + 1} \oplus \ldots$.
Then we define $\gr(a)$ to be the image of the projection of $a$ onto $A_i$ along $A_{i + 1} \oplus A_{i + 2} \oplus \ldots$.
In other words, we replace each element with its lowest homogeneous component.

Note that $\gr$ is not a homomorphism, it is not even a linear map.
However, it has two important propreties we state as a lemma.
\begin{lemma}\label{lem:gr_prop}
  \begin{enumerate}
  \item[]
      \item Let $a_1, \ldots, a_n \in A$ and let $p \in k[\mathbf{x}^{(\infty)}]$ be a differential monomial.
      Then
      \[
          p(\gr(a_1), \ldots, \gr(a_n)) \neq 0 \implies \gr(p(a_1, \ldots, a_n)) = p(\gr(a_1), \ldots, \gr(a_n)).
      \] 
      \item If $a_1, \ldots, a_n \in A$ are $k$-linearly dependent, then $\gr(a_1), \ldots, \gr(a_n)$ also are $k$-linearly dependent.
  \end{enumerate}
\end{lemma}

\begin{proof}
  To prove the first part, one sees that $p$ does not vanish on the lowest homogeneous parts of $a_1, \ldots, a_n$, so the homogeneity of the multiplication and derivation imply that taking the lowest homogeneous part commutes with applying $p$ for $a_1, \ldots, a_n$.
  
  To prove the second part, let $i$ be the lowest grading appearing among $a_1, \ldots, a_n$.
  Restricting to the component of this weight, one gets a linear relation for $\gr(a_1), \ldots, \gr(a_n)$.
\end{proof}

\begin{lemma}\label{lem:algebra_ineq}
  Let $A$ be a graded differential algebra as above.
  Consider elements $a_1, \ldots, a_n$ in $A$, and denote the algebras (not differential) generated by $a_1, \ldots, a_n$ and $\gr(a_1), \ldots, \gr(a_n)$ by $B$ and $B_{\gr}$, respectively.
  Then $\dim B_{\gr} \leqslant \dim B$.
\end{lemma}

\begin{proof}
  $B_{\gr}$ is spanned by all the monomials in $\gr(a_1), \ldots, \gr(a_n)$.
  We choose a basis in this spanning set, that is, we consider monomials $p_1, \ldots, p_N \in k[x_1, \ldots, x_n]$ such that
  \[
  p_1(\gr(a_1), \ldots, \gr(a_n)), \ldots, p_N(\gr(a_1), \ldots, \gr(a_n))
  \]
  form a basis of $B_{\gr}$.
  The first part of Lemma~\ref{lem:gr_prop} implies that 
  \[
    \gr(p_i(a_1, \ldots, a_n)) = p_i(\gr(a_1), \ldots, \gr(a_n)) \quad\text{ for every }1 \leqslant i \leqslant N.
  \]
  Then the second part of Lemma~\ref{lem:gr_prop} implies that $p_1(a_1, \ldots, a_n), \ldots, p_N(a_1, \ldots, a_n)$ are linearly independent.
  Since they belong to $B$, we have $\dim B \geqslant N = \dim B_{\gr}$.
\end{proof}


\begin{proof}[Proof of Proposition~\ref{prop:dim_bound}]
  Let $\Lambda$ and $\varphi$ be the exterior algebra and the homomorphism from Proposition~\ref{prop:ext_embedding}.
  Proposition~\ref{prop:ext_embedding} implies that $A_{m, h}$ is isomorphic to the subalgebra of $\Lambda$ generated by
   \[
     \sum\limits_{i = 0}^{m - 2}\eta_i \otimes \xi_i, \; \sum\limits_{i = 0}^{m - 2}(\eta_i \otimes \xi_i)', \; \sum\limits_{i = 0}^{m - 2}(\eta_i \otimes \xi_i)'', \;\ldots,\; \sum\limits_{i = 0}^{m - 2}(\eta_i \otimes \xi_i)^{(h)}.
   \]

  We define a grading on $\Lambda$ by setting the weights of $\eta_j^{(i)}$ and $\xi_j^{(i)}$ be equal to $i$ for every $i \geqslant 0$ and $0\leqslant j < m - 1$.
  $\Lambda$ becomes a graded algebra, and the derivation is homogeneous of weight one.
  
  We fix $h \geqslant 0$ and consider the following elements of $\Lambda$:
  \[
    \widetilde{\alpha}_{j, i} := (1 + \partial)^i \alpha_j \quad\text{ for }i \geqslant 0,\; 0 \leqslant j < m - 1 \text{ and } \alpha \in \{\eta, \xi\},
  \]
  where $\partial$ is the operator of differentiation.
  We introduce
  \[
    v_i := \sum\limits_{j = 0}^{m - 2} \widetilde{\eta}_{j, i} \otimes \widetilde{\xi}_{j, i} \quad\text{ for } 0 \leqslant i \leqslant h
  \]
  and let $Y_h$ be the algebra generated by $v_0, \ldots, v_h$.
  For every $0 \leqslant i \leqslant h$, we have $v_i^m = 0$, so $Y_h$ is spanned by the products of the form
  \[
    v_0^{d_0} v_1^{d_1} \ldots v_h^{d_h}, \quad\text{ where } 0 \leqslant d_0, \ldots, d_h < m.
  \]
  Therefore, $\dim Y_h \leqslant m^{h + 1}$.
  
  \textbf{Claim.} \emph{There is an invertible $(h + 1) \times (h + 1)$ matrix $M$ over $\QQ$ such that, for $u_0, \ldots, u_h$ defined by
  \begin{equation}\label{eq:def_u}
    (u_0, \ldots, u_h)^T := M(v_0, \ldots, v_h)^T,
  \end{equation}
  we have
  \[
    \gr(u_i) = \sum\limits_{j = 0}^{m - 2} (\eta_j \otimes \xi_j)^{(i)} \quad \text{ for every } 0 \leqslant i \leqslant h.
\]}
  
  We will first demonstrate how the proposition follows from the claim and then prove the claim.
  Since $M$ is invertible, $u_0, \ldots, u_h$ generate $Y_h$ as well.
  Since $\gr(u_0), \ldots, \gr(u_h)$ generate $A_{m, h}$, Lemma~\ref{lem:algebra_ineq} implies that $m^{h + 1} \geqslant \dim Y_h \geqslant \dim A_{m, h}$.
  
  Therefore, it remains to prove the claim.
  For every $0 \leqslant i \leqslant h$, we can write
  \[
    v_i = (1\otimes 1 + 1 \otimes \partial)^i (1 \otimes 1 + \partial\otimes 1)^i v_0 = (1\otimes 1 + 1 \otimes \partial + \partial \otimes 1 + \partial \otimes \partial)^i v_0.
  \]
  We set $u_i := (1 \otimes \partial + \partial \otimes 1 + \partial \otimes \partial)^i v_0$ for every $0 \leqslant i \leqslant h$.
  Note that, since $1 \otimes \partial + \partial \otimes 1$ is just the original derivation on $\Lambda$, we have
  \begin{equation}\label{eq:gr}
    \gr(u_i) = (1 \otimes \partial + \partial \otimes 1)^i v_0 = v_0^{(i)} = \sum\limits_{j = 0}^{m - 2} (\eta_j \otimes \xi_j)^{(i)}.
  \end{equation}
  By expanding the binomial $(1\otimes 1 + (1 \otimes \partial + \partial \otimes 1 + \partial \otimes \partial))^i$, we can write $v_i = \sum\limits_{j = 0}^i \binom{i}{j} u_j$.
  Then we have
  \begin{equation}
      (v_0, \ldots, v_h)^T = \widetilde{M} (u_0, \ldots, u_h)^{T},
  \end{equation}
  where $\widetilde{M}$ is the $(h + 1) \times (h + 1)$-matrix with the $(i, j)$-th entry being $\binom{i}{j}$.
  Since $\widetilde{M}$ is lower-triangular with ones on the diagonal, it is invertible.
  We set $M := \widetilde{M}^{-1}$.
  So we have $(u_0, \ldots, u_h)^T := M(v_0, \ldots, v_h)^T$ which, together with~\eqref{eq:gr} finishes the proof of the claim.
\end{proof}

By combining the proof of Proposition~\ref{prop:dim_bound} with Lemma~\ref{lem:ext_extra}, we can extend Proposition~\ref{prop:dim_bound} as follows.

\begin{corollary}\label{cor:dim_bound}
  Let $m, h, i$ be positive integers with $1 \leqslant i \leqslant m$.
  By $A_{(m, i), h}$ we denote the subalgebra of $k[x^{(\infty)}] / \langle x^i, (x^m)^{(\infty)}\rangle$ generated by the images of $x, x', \ldots, x^{(h)}$.
  Then
  \[
    \dim A_{(m, i), h} \leqslant i\cdot m^h.
  \]
\end{corollary}

\begin{proof}
  The proof will be a refinement of the proof of Proposition~\ref{prop:dim_bound}, and we will use the notation from there.
  Let $\pi$ be the canonical homomorphism $\pi \colon \Lambda \to \Lambda_i := \Lambda / \langle \xi_{i - 1} \otimes \eta_{i - 1}, \ldots, \xi_{m - 2} \otimes \eta_{m - 2}\rangle$.
  Since the ideal $\langle \xi_{i - 1} \otimes \eta_{i - 1}, \ldots, \xi_{m - 2} \otimes \eta_{m - 2}\rangle$ is homogeneous with respect to the grading on $\Lambda$, there is a natural grading on $\Lambda_i$.
  
  We have $A_{(m, i), h} \cong \pi(A_{m, h})$.
  Since $\pi$ is a homogeneous homomorphism, $\pi(A_{m, h})$ is generated by $\pi(\gr(u_0)), \ldots, \pi(\gr(u_h))$ from~\eqref{eq:def_u}, so $\dim A_{(m, i), h} = \dim \pi(A_{m, h}) \leqslant \dim\pi(Y_h)$.
  We observe that $\pi(v_0)^i = 0$, so $\pi(Y_h)$ is spanned by the products of the form 
  \[
    \pi(v_0)^{d_0} \pi(v_1)^{d_1} \ldots \pi(v_h)^{d_h},
  \]
where $0 \leqslant d_0 < i$ and $0 \leqslant d_1, \ldots, d_h < m$. Therefore, $\dim\pi(Y_h) \leqslant i\cdot m^h$.
\end{proof}


\subsection{Combinatorial properties of fair monomials}\label{sec:combi}

\begin{definition}[Non-overlapping monomials]
    We say that two monomials $m_1, m_2 \in k[x^{(\infty)}]$ \emph{do not overlap} if $\ord m_1 \leqslant \lord m_2$ or $\ord m_2 \leqslant \lord m_1$.
\end{definition}

\begin{lemma}\label{lem:nonoverlap}
    Let $m, i$ be integers with $0 \leqslant i \leqslant m$.
    Let $P \in \mathcal{F}_{i, m - i}$.
    Then there exist $P_1, \ldots, P_{i} \in \mathcal{F}$ and $P_{i + 1}, \ldots, P_{m} \in \mathcal{F}_s$ such that 
    \[
    P = P_1\cdot \ldots\cdot P_{m}\quad \text{ and, for every $1 \leqslant i < m$,}\quad \ord P_i \leqslant \lord P_{i + 1}.
    \]
\end{lemma}

\begin{remark}
Lemma~\ref{lem:nonoverlap} implies that the set $\mathcal{F}_{i - 1, m - i}$ from Theorem~\ref{thm:non-initial} coincides with the set of standard monomials conjectured by Afsharijoo in~\cite[Section~5]{Afsharijoo2021}.
\end{remark}

\begin{proof}
   Suppose that $P$ can be written as 
   \[
       P=(x^{(h_{1, 0})}\cdots x^{(h_{1, \ell_1})})\cdots (x^{(h_{m, 0})}\cdots x^{(h_{m, \ell_m})}),
    \]
 where each $(x^{(h_{i, 0})}\cdots x^{(h_{i, \ell_i})})$ belongs to $\mathcal{F}$ or $\mathcal{F}_s$, and $h_{1, 0} \leqslant h_{2, 0} \leqslant \cdots \leqslant h_{m, 0}$. 
 We first prove that we can make the product to be a product of non-overlapping monomials.
 
 Let us sort the orders $h_{1, 0}, h_{1, 1}, \ldots, h_{m, \ell_m}$ in the ascending order: 
 \[
     \{(r_{1, 0},\dots,r_{1, \ell_1}); (r_{2, 0},\ldots,r_{2, \ell_2});\ldots;(r_{m, 0},\ldots,r_{m, \ell_m})\}.
 \]
 
 \textbf{Claim :} \emph{For all $0\leqslant i\leqslant m$, $h_{i, 0} \leqslant r_{i, 0}$.}\\
 In the whole list of $h_{i, j}$'s, all the numbers to the right from $h_{i, 0}$ are $\geqslant h_{i, 0}$.
 Therefore, after sorting, $h_{i, 0}$ will either stay or move to the left.
 Thus, $h_{i, 0} \leqslant r_{i, 0}$, so the claim is proved.
 
 Hence if $x^{(h_{i, 0})}\cdots x^{(h_{i, \ell_i})}$ was a fair (resp., strongly fair) monomial then $x^{(r_{i, 0})}\cdots x^{(r_{i, \ell_i})}$ is a fair (resp., strongly fair) monomial.

 Now we will move all the strongly fair monomials to the right in the decomposition of $P$.
 We first prove that, for every $Q = Q_1 Q_2$ such that $Q_1 \in \mathcal{F}_s$, $Q_2\in\mathcal{F}$, and $\ord Q_1 \leqslant \lord Q_2$, there exist $\widetilde{Q}_1 \in \mathcal{F}_s$ and $\widetilde{Q}_2\in\mathcal{F}$ such that $Q = \widetilde{Q}_1\widetilde{Q}_2$ and $\ord \widetilde{Q}_1 \leqslant \lord \widetilde{Q}_2$. 
  Let 
 \[
 Q_1 = x^{ ( h_{1, 0} ) } \cdots x^{ ( h_{1, \ell_1} ) }, \quad Q_2 = x^{ ( h_{2, 0} ) } \cdots x^{ ( h_{2, \ell_2} ) }
 \]
where $ \ell_1 < h_{1, 0}$ and $\ell_2 \leqslant h_{2, 0}$. 
If $\ell_2 < h_{2, 0}$ then $Q_2 \in \mathcal{F}_s$, so we are done. 
Otherwise, $ \ell_1 + 1 \leqslant h_{1, 0}$ implies that $Q_1 x^{ ( h_{1, 0} )}$ is a fair monomial, and $\ell_2 - 1 < h_{2, 0}$ implies that $\tfrac{Q_2}{x^{ ( h_{1, 0} )}} \in \mathcal{F}_s$. 
Thus, we can take $\widetilde{Q}_1 := Q_1 x^{ ( h_{1, 0} )}$ and $\widetilde{Q}_2 := \tfrac{Q_2}{x^{ ( h_{1, 0} )}}$.

Applying the described transformation while possible to the non-overlapping decomposition of $P$, one can arrange that the last $m - i$ components are strongly fair.
 \end{proof}
 
\begin{proposition}\label{prop:count}
    For every positive integers $m, h, i$ with $0 \leqslant i \leqslant m$, the cardinality of $\mathcal{F}_{i, m - i} \cap k[x^{(\leqslant h)}]$ is equal to $(i + 1)\cdot (m + 1)^h$.
\end{proposition}

The proof of the proposition will use the following lemma.
\begin{lemma}\label{lem:m2}
  For every integers $h$ and $d$, we have
  \[
    |\{P \mid P \in \mathcal{F} \cap k[x^{\leqslant h}] \text{ and }\deg P = d\}| = \binom{h + 1}{d}.
  \]
  If one replaces $\mathcal{F}$ with $\mathcal{F}_s$, the cardinality will be $\binom{h}{d}$.
\end{lemma}

\begin{proof}
    Let $x^{(h_0)} \ldots x^{(h_\ell)} \in \mathcal{F}$ such that $\ell \leqslant h_0 \leqslant \ldots \leqslant h_\ell$.
    We define a map
    \[
      (h_0, \ldots, h_\ell) \mapsto (h_0 - \ell, h_1 - \ell - 1, \ldots, h_\ell).
    \]
    The map assign to the orders of a monomial in $\mathcal{F} \cap k[x^{\leqslant h}]$ a list of strictly increasing non-negative integers not exceeding $h$.
    A direct computation shows that this map is a bijection.
    Since the number of such sequences of length $d$ is equal to the number of subsets of $[0, 1, \ldots, h]$ of cardinality $d$, the number of monomials is $\binom{h + 1}{d}$.
    
    The case of $\mathcal{F}_s$ is analogous with the only difference that the subset will be in $[1, 2, \ldots, h]$ thus yielding $\binom{h}{d}$.
\end{proof}

\begin{proof}[Proof of Proposition~\ref{prop:count}]
   We will prove the proposition by induction on $m$.
   For the base case, we have $\mathcal{F}_{0, 0} = \{1\}$, so the statement is true.
   
   Consider $m > 0$ and assume that for all smaller $m$ the proposition is proved.
   We fix $0 \leqslant i \leqslant m$.
   Consider a monomial $P \in \mathcal{F}_{i, m - i} \cap k[x^{(\leqslant h)}]$, let $P_1\cdot \ldots \cdot P_{m}$ be a decomposition from Lemma~\ref{lem:nonoverlap} with $\deg P_{m}$ being as large as possible.
   We denote $\operatorname{tail}P := P_{m }$ and $\operatorname{head}P := P_1\ldots P_{m - 1}$.
   
   We will show that the map $P \to (\operatorname{head} P, \operatorname{tail} P)$ defines a bijection between $\mathcal{F}_{i, m - i}$ and
   \begin{align}
   \begin{split}\label{eq:bijection}
       \text{for } i < m\colon & \{(Q_0, Q_1) \in \mathcal{F}_{i, m - i - 1} \times \mathcal{F}_s \mid \ord Q_0 \leqslant \deg Q_1\},\\
       \text{for } i = m\colon & \{(Q_0, Q_1) \in \mathcal{F}_{m - 1, 0} \times \mathcal{F} \mid \ord Q_0 < \deg Q_1\}.
   \end{split}
   \end{align}
   We will prove the case $i < m$, the proof in the case $i = m$ is analogous.
   First we will show that, for every $P \in \mathcal{F}_{i, m - i}$, we have $\ord \operatorname{head} P \leqslant \deg \operatorname{tail} P$.
   Assume the contrary, and let $\ell := \ord \operatorname{head} P > \deg \operatorname{tail}P$.
   Then we will have
   \[
     \lord (x^{(\ell)} \operatorname{tail}P) \geqslant \min(\ell, \lord\operatorname{tail}P) = \ell \geqslant \deg (x^{(\ell)}\operatorname{tail}P).
   \]
    This implies that $x^{(\ell)} \operatorname{tail}P \in \mathcal{F}_s$.
    Thus, in the decomposition of Lemma~\ref{lem:nonoverlap} we could have taken $P_{m}$ to be $x^{(\ell)} \operatorname{tail}P$.
    This contradicts the maximality of $\deg \operatorname{tail} P$.
     In the other direction, if $Q_0 \in \mathcal{F}_{i, m - i - 1 }$ and $Q_1 \in \mathcal{F}_s$ such that $\ord Q_0 \leqslant \deg Q_1$, then $Q_0Q_1 \in \mathcal{F}_{i, m - i}$.
     Moreover, since $x^{(\ord Q_0)} Q_1 \not\in \mathcal{F}$, we have $\operatorname{tail}(Q_0 Q_1) = Q_1$.
   
     We will now use the bijection~\eqref{eq:bijection} to count the elements in $\mathcal{F}_{i, m - i} \cap k[x^{(\leqslant h)}]$. For $i < m$:
   \begin{align*}
     |\mathcal{F}_{i, m - i} \cap k[x^{(\leqslant h)}]| &= \sum\limits_{\ell = 0}^{h} |\mathcal{F}_{i, m - i - 1} \cap k[x^{(\leqslant \ell)}]| \cdot |\{Q_1 \in \mathcal{F}_s\cap k[x^{(\leqslant h)}] \mid \deg Q_1 = \ell\}|\\
     \text{(by Lemma~\ref{lem:m2})}&= \sum\limits_{\ell = 0}^h (i + 1) \cdot m^{\ell} \binom{h}{\ell} = (i + 1)\cdot (m + 1)^h.
   \end{align*}
     For $i = m$:
     \begin{align*}
     |\mathcal{F}_{m, 0} \cap k[x^{(\leqslant h)}]| &= \sum\limits_{\ell = 0}^{h + 1} |\mathcal{F}_{m - 1, 0} \cap k[x^{(< \ell)}]| \cdot |\{Q_1 \in \mathcal{F}_s\cap k[x^{(\leqslant h)}] \mid \deg Q_1 = \ell\}|\\
     \text{(by Lemma~\ref{lem:m2})}&= \sum\limits_{\ell = 0}^{h + 1} m^{\ell} \binom{h + 1}{\ell} = (m + 1)^{h + 1}.
   \end{align*}
   Thus, the proposition is proved.
\end{proof}


\subsection{Lower bounds for the dimension}\label{sec:lower}

\begin{notation}\label{Notation}
      For a differential polynomial $P \in k[\mathbf{x}^{(\infty )}]$ and $1 \leqslant i \leqslant n$, we define
  \begin{itemize}
      \item $\tord_{x_i} P$ to be the \emph{total order} of $P$ in $x_i$, that is, the largest sum of the orders of the derivatives of $x_i$ among the monomials of $P$;
      \item $\deg_{x_i^{(\infty)}} P$ to be the \emph{total degree} of $P$ with respect to the variables $x_i, x_i', x_i'', \ldots$.
      \item We fix a monomial ordering $\prec$ on $k[\mathbf{x}^{(\infty)}]$ defined as follows.
      To each differential monomial $M = x_{i_0}^{(h_0)} x_{i_1}^{(h_1)} \cdots x_{i_\ell}^{(h_\ell)}$ with $(h_0, i_0) \preceq_{\lex} (h_1, i_1) \preceq_{\lex} \ldots \preceq_{\lex} (h_\ell, i_\ell)$, we assign a tuple
      \[
        (\ell,\; h_{\ell},\; h_{\ell - 1},\; \ldots,\; h_0,\; i_\ell,\; i_{\ell - 1},\; \ldots,\; i_0),
      \]
      and compare monomials by comparing the corresponding tuples lexicographically.
  \end{itemize}
\end{notation}

\begin{definition}[Isobaric ideal]
  An ideal $I \subset k[x^{(\infty)}]$ is called \emph{isobaric} if it can be generated by isobaric polynomials, that is, polynomials with all the monomials having the same total order.
\end{definition}

\begin{proposition}\label{prop:m2}
For $i = 1, 2$, the elements of $\mathcal{F}_{i - 1, 2 - i}$ are the standard monomials  modulo $\langle (x^2)^{(\infty)}, x^i\rangle$. 
\end{proposition}
\begin{proof}
    We use Proposition~\ref{prop:ext_embedding} to obtain
    the differential homomorphism $\varphi\colon k[x^{(\infty)}] \to \Lambda$ defined by $\varphi(x)=\eta \otimes\xi$ (we will use $\eta$ and $\xi$ instead of $\eta_0$ and $\xi_0$ for brevity).
    Let $\tilde{\varphi}$ be the composition of $\varphi$ with the projection onto $\Lambda /\langle\eta\otimes\xi\rangle$. 
    We will prove the proposition for the elements in $\mathcal{F}_{1, 0}$, the other case can be done in the same way by replacing $\varphi$ with $\tilde{\varphi}$.

    Let $X = x^{(h_0)}\cdots x^{(h_\ell)}$, $h_0 \leqslant h_1 \leqslant \ldots \leqslant h_\ell$ be an element of $\mathcal{F}_{1, 0}$.
    We will show that a summand 
    \begin{equation}\label{eq:BX}
      B(X) := (\eta^{(h_0-\ell)}\wedge \eta^{(h_1-(\ell-1))}\wedge\ldots\wedge \eta^{(h_\ell)}) \otimes (\xi^{(\ell)} \wedge \xi^{(\ell - 1)} \wedge \ldots \wedge \xi' \wedge \xi)
    \end{equation}
    appears in  $\varphi(X)$ with nonzero coefficient. 
    We will prove this by induction on $\ell$. 
    The base case $\ell = 0$ is trivial, let $\ell > 0$.
    Since $\eta^{(h_0 - \ell)}$ may come only from one of the occurrences of $x^{(h_0)}$ in $X$, we must take $\eta^{(h_0 - \ell)} \otimes \xi^{(\ell)}$ from one of the $x^{(h_0)}$'s.
    Therefore, the coefficient at $B(X)$ in $\varphi(X)$ is $\deg_{x^{(h_0)}}X$ times the coefficient at $B(X / x^{(h_0)})$ in $\varphi(X / x^{(h_0)})$ which is nonzero by the induction hypothesis.

    Let $Y : = x^{(s_0)}\ldots x^{(s_ {\ell'})}$ be a monomial such that $ Y \prec X$.
    We will prove by contradiction that $B(X)$ does not appear in $\varphi(Y)$. 
    If it does, then $\deg (X) = \deg(Y) = \ell + 1 = \ell' + 1$. 
    Moreover, there exists a permutation $\sigma$ of $\{ 0 , 1 , \ldots , \ell \}$ such that 
    \[
    s_i - \sigma (i) = h_i - ( \ell - i ) \quad \text{for every }0 \leqslant i \leqslant \ell.
    \]
    The inequality $s_\ell \leqslant h_\ell$ implies $\sigma (\ell) = 0$ and, thus, $s_\ell = h_\ell$. 
    Therefore, $s_{\ell - 1} \leqslant h_ {\ell -1} $, which implies $ \sigma( \ell - 1) = 1$ and, thus, $s_{\ell - 1} = h_{\ell -1}  $.
    Continuing this way, we show that:
    \[
    \forall \, 0 \leq i \leq \ell \quad s_i = h_i 
    \]
    which is contradicting $Y \prec X $. 
    Thus $B(X)$ cannot appear in the $\varphi(Y)$. 
    
    Assume that $X \in \init_\prec \langle x^2\rangle^{(\infty)}$.
    Then there exist monomials $P_1,\ldots P_N$ such that $P_j \prec X$ for all $1 \leq  j \leq N$ and 
    \[
      X-\sum\limits_{j=1}^N \lambda P_j \in \langle x^2\rangle^{(\infty)}.
    \]
    Hence $\varphi(X)-\sum_{j=1}^N \lambda_j\varphi(P_j)=0$. Since $P_j \prec X$ for all $1 \leq  j \leq N$, $B(X)$ cannot be canceled in $\varphi(X)-\sum_{j=1}^N \lambda_j\varphi(P_j)$ which is a contradiction. Therefore $X$ is a standard monomial.
\end{proof}

\begin{lemma}\label{Buch}
Let $I_1\subset k[y_1^{(\infty)}], \ldots, I_s\subset k[y_s^{(\infty)}]$ be ideals, and by $M_i$ we denote the set of the standard monomials modulo $I_i$ w.r.t degree lexicographic ordering for $1 \leqslant i \leqslant s$.
Then the standard monomials w.r.t the ordering $\prec $ (see Notation \ref{Notation}) modulo $\langle I_1,\ldots ,I_s\rangle \subset k[y_1^{(\infty)},\ldots, y_s^{(\infty)}]$ are 
  \[
      M_1\cdot M_2\cdots M_s:=\{m_1m_2\cdots m_s\,|\, m_1\in M_1,\ldots, m_s\in M_s\}.
  \]
\end{lemma}
\begin{proof}
For each $I_i$, consider the reduced Gr\"obner basis $G_i$ of $I_i$ w.r.t. the degree lexicographic ordering. 
For each pair $f, g \in G := G_1\cup G_2\cup\ldots \cup G_s$, their S-polynomial is reduced to zero by $G$
\begin{itemize}
    \item if $f, g$ belong to the same $G_i$, due to the fact that $G_i$ is a Gr\"obner basis;
    \item otherwise, by the first Buchberger criterion (since $f$ and $g$ have coprime leading monomials).\qedhere
\end{itemize}
\end{proof}

\begin{proposition}\label{prop:indep}
Let $I_1\subset k[y_1^{(\infty)}], \ldots, I_s\subset  k[y_s^{(\infty)}]$  be homogeneous and isobaric ideals (not necessarily differential).
By $M_i$ we denote the set of standard monomials modulo $I_i$ w.r.t the degree lexicographic ordering for $1 \leqslant i \leqslant s$. 
We define a homomorphism (not necessarily differential)
\[
\varphi\colon k[x^{(\infty)}]\to k[y_1^{(\infty)},\ldots,y_{s}^{(\infty)}]/\langle I_1,\ldots ,I_s\rangle
\]
 by $\varphi(x^{(k)}) := y_1^{(k)} + \ldots + y_s^{(k)}$ and denote $I := \operatorname{Ker}(\varphi)$. 
 Then the elements of 
 \begin{equation}\label{eq:defM}
   M:=\{m_1 \ldots m_s \mid \forall 1 \leqslant i \leqslant s\colon m_i \in M_i \;\text{ and }\;\forall\; 1\leqslant j < s\colon \ord m_j \leqslant \lord m_{j + 1} \}
   \end{equation}
   are standard monomials modulo $I$ w.r.t. the  ordering $ \prec $ (but maybe not all the standard monomials). 
\end{proposition}

\begin{proof} 
    Consider a monomial $P=x^{(h_0)}\cdots x^{(h_\ell)} \in M$, and fix a representation $P = m_1(x), \ldots, m_{s}(x)$ as in~\eqref{eq:defM}.
    Assume that $P$ is a leading monomial of $I$. Then there exist monomials $P_1, \ldots, P_N$ such that 
    \[
      P - \sum\limits_{j = 1}^N \lambda_j P_j \in \operatorname{Ker} \varphi \quad \text{ and }\quad \forall\; 1 \leqslant j \leqslant N\colon P_j\prec P.
    \]
    Then $\varphi(P)-\sum \lambda_j\varphi(P_j)\in \langle I_1,\ldots I_s\rangle $.
    We define $m := m_1(y_1) m_2(y_2)\ldots m_s(y_s)$.

    \textbf{Claim.} \emph{For every monomial $\widetilde{m} \neq m$ in $\varphi(P)$, there exists $1 \leqslant j \leqslant s$ such that either $\deg_{y_j^{(\infty)}} m \neq \deg_{y_j^{(\infty)}} \widetilde{m}$ or $\tord_{y_j} m \neq \tord_{y_j} \widetilde{m}$.}

    Assume the contrary that there exists $\widetilde{m}$ such that, for every $1 \leqslant j \leqslant s$,  $d_i := \deg_{y_j^{(\infty)}} m = \deg_{y_j^{(\infty)}} \widetilde{m}$ and $\tord_{y_j} m = \tord_{y_j} \widetilde{m}$.
    We write $\widetilde{m} = \widetilde{m}_1(y_1)\ldots \widetilde{m}_s(y_s)$.
    Let $1 \leqslant j \leqslant s$ be the largest index such that $m_j \neq \widetilde{m_j}$.
    Since $m_j$ contains $d_j$ largest derivatives in $m_1(x)\ldots m_j(x) = \widetilde{m}_1(x)\ldots \widetilde{m}_j(x)$ and has the same total order as $\widetilde{m}_j$, we conclude that $m_j = \widetilde{m}_j$.
    Thus, the \textbf{claim} is proved.
    
     We write the homogeneous and isobaric component of $\sum_{j = 1}^N\lambda_j\varphi(P_j)$ of the same degree and total order in $y_i$ as $m$ for every $1 \leqslant i \leqslant s$ as $\sum_{i = 1}^M \mu_i R_i$, where $R_i$ is a differential monomial and $\mu_i \in k$ for every $1 \leqslant i \leqslant M$.
     Then such a homogeneous and isobaric component of $\varphi(P)-\sum_{j = 1}^N \lambda_j \varphi(P_j)$ is  
     $Q := m - \sum_{i = 1}^M \mu_i R_i$
     due to the claim.
     Since, for every $1 \leqslant i \leqslant s$, $I_s$ is homogeneous and isobaric, $Q \in \langle I_1, \ldots, I_s\rangle$.
     
     Note that for every $1\leqslant i \leqslant M$, $R_{i}$ is a summand of $\varphi(P_j)$ for some $1 \leqslant j \leqslant N$.  
     Thus, if $P_j = x^{(s_0)} \ldots x^{(s_\ell)}$, then the derivatives that appear in the monomial $R_i$ are of orders $s_0 , \ldots , s_\ell$. 
     Hence $P_j \prec P$ implies $R_j \prec m$. 
     Therefore $m$ is the leading monomial of $Q$ contradicting Lemma~\ref{Buch}.
\end{proof}

\begin{corollary}\label{cor:linear_independence}
The elements of  $\mathcal{F}_{i - 1, m - i}$ are standard monomials modulo $\langle x^i, (x^m)^{(\infty)} \rangle$.
\end{corollary}
\begin{proof}
  We will use Proposition~\ref{prop:indep}.
  Consider the ideals 
  \[
  I_1 = \langle y_1^2 \rangle^{(\infty)}, \ldots ,I_{i-1}= \langle y_{i-1}^2\rangle^{(\infty)}, I_i=\langle y_i, (y_i^2)^{(\infty)}\rangle,\ldots, I_{m-1}=\langle y_{m-1}, (y_{m-1}^2)^{(\infty)}\rangle.
  \]
  and define $\varphi$ as in Proposition~\ref{prop:indep}.
  Lemma~\ref{lem:ext_extra} implies that $\varphi((x^m)^{(k)}) = ((y_1+\ldots+y_{m - 1})^m)^{(k)}=0$ for every $k \ge 1$ and $\varphi(x^i)=(y_1+\ldots+y_{i - 1})^i=0$. 
  Therefore, $\langle (x^m)^{(\infty)},x^i\rangle \subset \operatorname{Ker}(\varphi)$.
  Proposition~\ref{prop:m2} implies that the standard monomials modulo $I_j$ are the fair monomials for $j < i$ and strongly fair monomials for $i \leqslant j$.
  Therefore, Proposition~\ref{prop:indep}  implies that $\mathcal{F}_{i - 1, m - i}$ are standard monomials modulo $\langle x^i, (x^m)^{(\infty)}\rangle$.
\end{proof}


\subsection{Putting everything together: proofs of the main results}

\begin{proof}[Proof of Theorem~\ref{thm:dimension}]
    Consider the images of $\mathcal{F}_{m - 1, 0} \cap k[x^{(\leqslant h)}]$ in $k[x^{(\infty)}] / \langle x^m \rangle^{(\infty)}$.
    By Corollary~\ref{cor:linear_independence}, they are linearly independent modulo $\langle x^m \rangle^{(\infty)}$.
    Then Proposition~\ref{prop:count} implies that the dimension of $k[x^{(\leqslant h)}] / \langle x^m\rangle^{(\infty)}$ is at least $m^{h + 1}$.
    Together with Proposition~\ref{prop:dim_bound}, this implies 
    $$\dim (k[x^{(\leqslant h)}] / \langle x^m\rangle^{(\infty)}) = m^{h + 1}.\qedhere$$
\end{proof}

\begin{proof}[Proof of Theorem~\ref{thm:non-initial}]
  Fix $h \geqslant 0$.
  Consider $\mathcal{F}_{i - 1, m - i} \cap k[ x^{(\leqslant h ) }]$.
  Combining Corollary~\ref{cor:linear_independence}, Corollary~\ref{cor:dim_bound}, and Proposition~\ref{prop:count}, we show that the image of this set in $k[x^{(\leqslant h)}] / \langle (x^m)^{(\infty)}, x^i \rangle$ forms a basis.
  Thus, the image of the whole $\mathcal{F}_{i - 1, m - i}$ is a basis of $k[x^{(\infty)}] / \langle (x^m)^{(\infty)}, x^i \rangle$. 
  Therefore, by Corollary \ref{cor:linear_independence}, $\mathcal{F}_{i - 1, m - i}$ coincides with the set of standard monomials modulo $\langle (x^m)^{(\infty)}, x^i \rangle$.
\end{proof}

\begin{proof}[Proof of Corollary~\ref{cor:non-initial}]
  Since the ideal $\langle x^i, (x^m)^{(\infty)} \rangle$ is generated by homogeneous and isobaric (that is, weight-homogeneous) polynomials, its Gr\"obner bases with respect to the purely lexicographic, degree lexicographic, and weighted lexicographic orderings coincide.
\end{proof}


\section{Computational experiments for more general fat points}\label{sec:compute}

In this section, we consider a more general case of a fat point in a $n$-dimensional space, not just on a line.
We used {\sc Macaulay2}~\cite{M2} and, in particular, package {\sc Jets}~\cite{galetto2021computing, jetsM2} to explore possible analogues of our Theorem~\ref{thm:dimension} for this more general case. 
A related Sage implementation for computing the arc space of an affine scheme with respect to a fat point can be found in \cite[Section 9]{STOUT} and \cite[Section 5.4]{StoutPhD}. 

Let $\mathbf{x} = (x_1, \ldots, x_n)$, and consider a zero-dimensional ideal $I \subset k[\mathbf{x}]$.
We will be interested in describing (in particular, in computing the dimension of the quotient ring) $I^{(\infty)} \cap k[\mathbf{x}^{(\leqslant h)}]$ for a positive integer $h$.
Since this ideal is the union of the following chain 
\[
  I^{(\leqslant 1)} \cap k[\mathbf{x}^{(\leqslant h)}] \subseteq I^{(\leqslant 2)} \cap k[\mathbf{x}^{(\leqslant h)}] \subseteq I^{(\leqslant 3)} \cap k[\mathbf{x}^{(\leqslant h)}] \subseteq \ldots
\]
and $k[\mathbf{x}^{(\leqslant h)}]$ is Noetherian, one can compute $I^{(\infty)} \cap k[\mathbf{x}^{(\leqslant h)}]$ by computing $I^{(\leqslant H)} \cap k[\mathbf{x}^{(\leqslant h)}]$ for large enough $H$.
But how to determine what $H$ is ``large enough''?
\begin{itemize}
    \item For the case $I = \langle x^m \rangle \subset k[x]$, the answer is given by our Theorem~\ref{thm:dimension}: if the dimension $k[x^{(\leqslant h)}] / (I^{(\leqslant H)} \cap k[x^{(\leqslant h)}])$ is equal to $m^{h + 1}$, then $I^{(\leqslant H)} \cap k[x^{(\leqslant h)}] = I^{(\infty)} \cap k[x^{(\leqslant h)}]$.
    \item For the case of general $I$, we take $H$ to be $1, 2, \ldots$, and we stop when we encounter $I^{(\leqslant H)} \cap k[\mathbf{x}^{(\leqslant h)}] = I^{(\leqslant H + 1)} \cap k[\mathbf{x}^{(\leqslant h)}]$.
    We conjecture that in this case $I^{(\leqslant H)} \cap k[\mathbf{x}^{(\leqslant h)}] = I^{(\infty)} \cap k[\mathbf{x}^{(\leqslant h)}]$ (see Question~\ref{q:stopping}) but, strictly speaking, we know only $I^{(\leqslant H)} \cap k[\mathbf{x}^{(\leqslant h)}] \subseteq I^{(\infty)} \cap k[\mathbf{x}^{(\leqslant h)}]$.
\end{itemize}

\subsection{Ideals $I = \langle x^m \rangle$}

For ideals of the form $\langle x^m \rangle$, the approach outlined above yields a complete algorithm to compute $I^{(\infty)} \cap k[x^{(\leqslant h)}]$ for any given $h$ and $m$.
We use it for computing examples of Gr\"obner bases for these ideals w.r.t. the lexicographic ordering.

{\small
\def\arraystretch{1.5}
\begin{table}[H]
\centering
\begin{tabular}{|l|l|}
\hline
 Ideal  & Gr\"obner basis   \\[0.8ex]
 \hline
 $\langle x^2 \rangle^{(\infty)} \cap k[x^{(\leqslant 2)}]$  &  $(x'')^4;\,  x' (x'')^2;\, (x')^2 x'';\, (x')^3;\, 2 x x'' + (x')^2;\, x x';\, x^2$  \\
 \hline
 $\langle x^3 \rangle^{(\infty)} \cap k[x^{(\leqslant 2)}]$  &  $(x'')^7;\, x' (x'')^5;\, (x')^2 (x'')^4;\, (x')^3 (x'')^2;\, (x')^4 x'';\, (x')^5;\, x (x'')^4 + 2 (x')^2 (x'')^3;$\\
 & $\, 3 x x' (x'')^2 + (x')^3 x'' ;\, 6 x (x')^2 x'' + (x')^4;\, x (x')^3;\, x^2 x'' + x (x')^2;\, x^2 x';\, x^3$\\
 \hline
\end{tabular}
\end{table}}

\subsection{General fat points}

In this subsection, we consider a general zero-dimensional $I \subset k[\mathbf{x}]$ with the zero set of $I$ being the origin.
We use the following algorithm following the approach described in the beginning of the section to obtain an upper bound of the dimensions of $k[\mathbf{x}^{(\leqslant h)}] / (I^{(\infty)} \cap k[\mathbf{x}^{(\leqslant h)}])$.

\begin{enumerate}[label = \textbf{(Step~\arabic*)}, leftmargin=17mm, itemsep=1pt]
    \item Set $H = 1$.
    \item While the dimension of $I^{(\leqslant H)} \cap k[\mathbf{x}^{(\leqslant h)}]$ is not zero or $I^{(\leqslant H)} \cap k[\mathbf{x}^{(\leqslant h)}] \neq I^{(\leqslant H + 1)} \cap k[\mathbf{x}^{(\leqslant h)}]$, set $H = H + 1$.
    \item Return $\dim \bigl( k[\mathbf{x}^{(\leqslant h)}] / (I^{(\leqslant H)} \cap k[\mathbf{x}^{(\leqslant h)}]) \bigr)$.
\end{enumerate}
We expect the resulting bound to be exact (see also Question~\ref{q:stopping}), for example, it is exact for the ideals $I = \langle x^m\rangle$.

Our implementation of this algorithm in Macaulay2 is available at \url{ https://mathrepo.mis.mpg.de/MultiplicityStructureOfArcSpaces}.
Table~\ref{table:general_good} below shows some of the results we obtained.
One can see that the computed dimensions form geometric series with the exponent being the multiplicity of the original ideal exactly as in Theorem~\ref{thm:dimension}.
\def\arraystretch{1.1}
\begin{table}[H]
\centering
\begin{tabular}{|l|l|l|l|l|}
\hline
 Ideal & $h = 0$ & $h = 1$ & $h = 2$ & $ h = 3$  \\[0.6ex]
 \hline
 $\langle x^2, y^2, xy \rangle$ & $3$ & $9$ &  $27$ & $81$ \\[0.6ex]
 \hline 
 $\langle x^2, y^2, xz, yz, z^2 - xy \rangle$ & $5$ & $25$ &  $125$ & -  \\[0.6ex]
 \hline
 $\langle x^3, y^2, x^2y \rangle$ & $5$ & $25$ &  $125$ & -  \\[0.6ex]
 \hline
 $\langle x^3, y^2, xy \rangle$ & $4$ & $16$ &  $64$ & $256$  \\[0.6ex]
 \hline
 $\langle x^3, y^3, x^2y \rangle$ & $7$ & $49$ &  $-$ & -  \\[0.6ex]
 \hline
  $ \langle x^4, y^4, x^2y^3\rangle $ &  14 & 196 & - & -\\[0.6ex]
  \hline
\end{tabular}
\caption{(Bounds for) the dimensions of the truncations of the arc space}\label{table:general_good}
\end{table}

However, we have also found ideals for which the generating series of the dimensions is definitely not equal to $\frac{m}{1 - mt}$, where $m$ is the multiplicity of the ideal.
We show some examples of this type in Table~\ref{table:general_bad}.
\begin{table}[H]
\centering
\begin{tabular}{|l|l|l|l||l|l|l|l|}
\hline
 Ideal & $h = 0$ & $h = 1$ & $h = 2$ & Ideal & $h = 0$ & $h = 1$ & $h = 2$ \\
 \hline
 $\langle x^3, y^3, xy \rangle$ & $5$ & $24$ &  $115$ &   $ \langle x^4, y^4, x^2y\rangle $ &  $10$ & $94$ & - \\
 \hline
 $\langle x^4, y^3, xy \rangle$ & $6$ & $33$ &  - & $\langle x^4, y^4, x^2y^2\rangle $ &  $12$ & $140$ & -  \\
 \hline
 $\langle x^4, y^3, x^2y \rangle$ & $8$ & $62$ &  -  & $ \langle x^4, y^6, x^2y^3\rangle $ &  $18$ & $320$ & - \\
 \hline
  $ \langle x^4, y^4, xy\rangle $ &  $7$ & $42$ & - & & & & \\
  \hline
  \end{tabular}
  \caption{(Bounds for) the dimensions of the truncations of the arc space}\label{table:general_bad}
\end{table}

Note that while Table~\ref{table:general_good} gives only indication that the generating series of the multiplicities for these ideals may be $\frac{m}{1 - mt}$, Table~\ref{table:general_bad} gives a proof that this is not the case for all the fat points.

\subsection{Open questions}

Based on the results of the computational experiments, we formulate several open questions.

\begin{question}\label{q:stopping}
  Let $I \subset k[\mathbf{x}]$ be a zero-dimensional ideal with $V(I)$ being a single point. 
  Is it true that, for every integer $h$:
  \[
    \bigl( I^{(\leqslant H)} \cap k[\mathbf{x}^{(\leqslant h)}] = I^{(\leqslant H + 1)} \cap k[\mathbf{x}^{(\leqslant h)}] \bigr) \implies \bigl( I^{(\leqslant H)} \cap k[\mathbf{x}^{(\leqslant h)}] = I^{(\leqslant \infty)} \cap k[\mathbf{x}^{(\leqslant h)}] \bigr)?
  \]
  Does this statement remain true if we drop the assumption $|V(I)| = 1$?
\end{question}

\begin{question}
  Let $I \subset k[\mathbf{x}]$ be a zero-dimensional ideal with $V(I)$ being a single point of multiplicity $m$. 
  Is it true that
  \[
    \lim\limits_{h \to \infty} \frac{\dim k[\mathbf{x}^{(\leqslant h)}] / I^{(\infty)} }{m^{h + 1}} = 1?
  \]
\end{question}

\begin{question}
  Let $I \subset k[\mathbf{x}]$ be a zero-dimensional ideal with $V(I)$ being a single point of multiplicity $m$. 
  Under which conditions it is true that
  \[
    \sum\limits_{h = 0}^{\infty} (\dim k[\mathbf{x}^{(\leqslant h)}] / I^{(\infty)} ) \cdot t^h = \frac{m}{1 - mt}?
  \]
  More generally, what information about the corresponding scheme can be read off the above generating series?
\end{question}


\subsection*{Acknowledgements}

The authors are grateful to Joris van der Hoeven, Hussein Mourtada, Bernd Sturmfels, Dmitry Trushin, and the referees for helpful discussions. We thank Yassine El Maazouz and Claudia Fevola for their support with making the Mathrepo webpage.
GP was partially supported by NSF grants DMS-1853482, DMS-1760448, and  DMS-1853650, by the Paris Ile-de-France region, and by the MIMOSA project funded by AAP INS2I CNRS.

\bibliographystyle{abbrvnat}
\bibliography{bibdata}
\end{document}